\documentclass[preprint,review,12pt]{elsarticle}
\usepackage{amssymb}
\usepackage[active]{srcltx}
\usepackage{a4wide}
\usepackage{graphicx,palatino,pifont,times}
\usepackage{oldgerm}
\usepackage{amsmath,amsthm}
\usepackage{amssymb}
\usepackage{amstext}
\usepackage{multirow}
\usepackage{epsfig}
\usepackage{enumerate}

\newtheorem{case 1}{case}
\newtheorem{thm}{Theorem}[section]

\newtheorem{lemma}[thm]{Lemma}

\newtheorem{theorem}[thm]{Theorem}

\theoremstyle{definition}
\newtheorem{definition}[thm]{Definition}

\theoremstyle{remark}
\newtheorem{remark}[thm]{Remark}
\newtheorem{example}{Example}
\numberwithin{equation}{section}

\catcode`\@=11
\def\theequation{\@arabic{\c@section}.\@arabic{\c@equation}}
\catcode`\@=12
\biboptions{sort}
\journal{Fixed point theory
}
\begin{document}
	\begin{frontmatter}
		\title{
			Redefining fractals through Suzuki Contraction.
		}
		\author[label1]{Kanase Pankaj Popatrao}
		\cortext[cor1]{Email address:pankajkanaseacademic@gmail.com}
		%\ead{chand@iitm.ac.in}
		\address[label1]{Department of Mathematics, SRM Institute of Science and Technology,Kattakulathur Chennai 603203, India.}
		\begin{abstract}
			There has been a significant effort in recent years to generalize the traditional concept of iterated function systems (IFS).In this article, we proposed Suzuki contraction in hyperspace and finding out the fixed point for Hutchinson mapping, which is called a deterministic fractal. The deterministic fractal for such a Suzuki contraction mapping is shown to exist and to be unique. We propose the Suzuki IFS (SIFS) in the literature for fractal creation based on this conclusion.
		\end{abstract}
	\end{frontmatter}
	\section{Introduction and preliminaries}
	\subsection{Introduction}
	Fractal theory, pioneered by Benoit Mandelbrot in $1975$, focuses on identifying patterns in nature's intricate, self-similar, and irregular shapes. The concept of self-similarity is fundamental to fractals, where patterns repeat themselves at different scales. Michael Barnsley's work on iterated function systems (IFS) in $1981$ provided a mathematical framework for generating self-similar fractals. Hutchinson's contribution was crucial in developing the mathematical techniques for constructing these fractals using an iterated function system. Fixed point theory, a branch of mathematics, plays a significant role in understanding and analyzing iterated function systems. It helps in determining the stability and behavior of fractal sets generated by IFS. The IFS (Iterated Function Systems) theory framework pioneered by Hutchinson has seen significant expansion to encompass broader spaces and generalized contractions. Hata extended the framework by incorporating condition /phi functions. Fernau introduced the concept of limitless IFSs, while Gwońzd´z-Lukowska and Jachymski, Mauldin and Urba´nski, Klimek and Kosek, Le´sniak, and Secelean have all contributed noteworthy works in this field. Secelean specifically delved into the study of countable iterated function systems on a compact metric space. These contributions collectively enrich and advance the understanding of iterated function systems and their applications.
	Iterated function systems have found applications in various fields, including Stochastic growth models, Image Compression,
	Signal and Image Processing, Terrain Generation, Pattern Recognition, Data Compression, Chaos Theory and Dynamical Systems, Financial Markets, approximation theory, study of bio-electric recordings. They provide a powerful tool for generating and analyzing complex structures that exhibit fractal properties. For further details on the applications and implications of iterated function systems and fractal geometry in applied sciences, one can refer to relevant literature and research papers \\
	In $1922$, Banach introduced a novel concept in his doctoral thesis concerning fixed point theory. He proposed that every Banach contraction on a complete metric space possesses a unique fixed point, often referred to as the Banach Contraction Theorem (BCT) or Banach Fixed Point Theorem (BCT). Since then, Some of the earliest notable generalizations of the BCT can be found in.\\ 
	In $2008$ 
	In comparison to traditional Iterated Function Systems (IFS), SIFS variation offers several distinct advantages, including increased flexibility, enhanced scalability, and improved fidelity. By incorporating elements of randomness and nonlinearity, Suzuki's IFS enables the generation of fractals with greater diversity and complexity, pushing the boundaries of mathematical creativity and exploration. SIFS inspires creativity, fosters innovation, and unlocks new possibilities for artistic expression and scientific inquiry.
	\section{\textbf{Main results}}
	
	\begin{definition}\label{D1}\cite{Suzuki 2008}
		Let  $(S,d)$  be a complete metric space and let $T$ be a self-mapping on $S$.
		Define a non-increasing mapping  $Q:\big[0,1\big)\to\big(Q(m),1\big]$ by
		\begin{equation*}
			Q(m)= \left\{
			\begin{array}{ll}
				1~~&if\,\qquad\quad~~~ 0\leq m \leq \frac{(\sqrt{5}-1)}{2}, \\
				\frac{(1-m)}{m^2}&if \,\quad~~~\frac{(\sqrt{5}-1)}{2}\leq m \leq \frac{1}{\sqrt2}, \\
				\frac{1}{(1+m)}&if \,\qquad\quad ~\frac{1}{\sqrt2}\leq m<1.\\
			\end{array}\right
			.\end{equation*}
		If there exists $ m \in \big[0,1\big)$ such that 
		\begin{equation}\label{Eq1}
			Q(m)~d(a,Ta)\leq d(a,b)~\implies~ d(Ta,Tb)\leq m ~d(a,b)~\forall~a,b \in S,
		\end{equation} then $T$ is said to be a Suzuki contraction mapping with contractivity factor $m$.
	\end{definition}
	\begin{remark} \label{R1}
		Every Banach contraction is a Suzuki contraction but every Suzuki contraction need not be a Banach contraction.  
	\end{remark}
	\begin{example}\label{E1}
		Let $S=\{(0,0),(4,0),(0,4),(4,5),(5,4)\}$ is complete metric space with the metric $d((x_1,x_2),(y_1,y_2))=\lvert x_1-y_1\rvert+\lvert x_2-y_2\rvert$ and let mapping $T$ be defined by 
		\begin{equation}
			T(x_1,x_2)=\left\{
			\begin{array}{ll}
				(x_1,0)& x_1\leq x_2, \\
				(0,x_2)&x_2<x_1.\\
			\end{array}\right.
		\end{equation} Then $T$ is a Suzuki contraction but not a Banach contraction because $T$ satisfies the condition $(\ref{Eq1})$ but not Banach condition.
		\begin{proof}
			suppose $x=(5,4), y=(4,5)$
			\begin{equation*}
				\begin{split}
					&Q(m) d(x,Tx)\leq d(x,y) \implies d(Tx,Ty)\leq md(x,y)\\&
					Q(m)d((5,4),T(5,4))\leq d((5,4),(4,5))\implies d(T(5,4),T(4,5))\leq md((5,4),(4,5))\\&Q(m)d((5,4),(0,4))\leq d((5,4),(4,5))\implies d((0,4),(4,0))\leq md((5,4),(4,5))\\& Q(m)\lvert 5-0\rvert+\lvert4-4\rvert \leq \lvert 5-4\rvert+\lvert4-5\rvert \implies\lvert0-4\rvert+\lvert4-0\rvert\leq m \lvert 5-4\rvert+\lvert4-5\rvert\\& Q(m)5\leq2 \implies 8\leq m2
				\end{split}
			\end{equation*}there doesnot exist $m\in\big[0,1\big)$ so that $T$ is not Banach at  $x=(5,4), y=(4,5)$ and for $x=(4,5), y=(5,4)$.
		\end{proof}
		\textbf{the following example shows why we said Suzuki working on reducing domain comparatively Banach.}
		\begin{example}Let $S=\big[0,1\big]$ be metric space with usual metric $d$. define mapping $T$ on $S$ by 
			\begin{equation}
				T(a)=\frac{a}{2}
			\end{equation} 
			above mapping is Banach contraction as well as Suzuki contraction.
		\end{example}
		\begin{proof}
			\begin{equation*}
				\begin{split}
					&Q(m)d(a,Ta)\leq d(a,b)\implies d(Ta,Tb)\leq md(a,b)\\&Q(m)\lvert a-\frac{a}{2}\rvert\leq \lvert a-b\rvert \implies \lvert \frac{a}{2}-\frac{b}{2}\rvert\leq m \lvert a-b\rvert\\& Q(m)\lvert\frac{a}{2}\lvert\leq \lvert a-b\rvert \implies Q(m) \lvert a-b\rvert\leq m \lvert a-b\rvert
				\end{split}
			\end{equation*}
			Then $T$ is suzuki contraction with contractivity factor $m\in \big[Q(m),1\big)$.
		\end{proof}
		\begin{remark}
			from the above example, we concluded that Suzuki contracton holds where this inequality satisfies  $ Q(m)\lvert\frac{a}{2}\lvert\leq \lvert a-b\rvert$ inequality in domain. no need to find all elements in the domain. then also having a unique fixed point for that Suzuki contraction. 
		\end{remark}
	\end{example} 
	
	%	\begin{remark}
		%		If $m\to1$ then Suzuki contraction is contractive.
		%	\end{remark} 
	\begin{remark}\label{R2}\cite{Pant 2022} Every Suzuki contraction need not be continuous.
	\end{remark}
	The BCP was forcefully generalized by Suzuki as follows:
	\begin{theorem}\cite{Suzuki 2008}\label{th1}
		Let $T$ be a Suzuki contraction mapping on the complete metric space $(S,d)$
		then $T$ has a unique fixed point $\bar{a}$. Moreover, $\underset{t\to\infty}\lim T^{\circ t}(a)=\bar{a}~ \forall~ a\in S $. 
	\end{theorem}
	%	\begin{example}
		%	Let $X=\{0,1,2\}$ be metric space endowed with following metric
		%	\begin{equation*}
			%		\begin{split}
				%			&d(0,0)=d(1,1)=d(2,2)=0,\\&d(0,1)=d(1,0)=5,\\&d(1,2)=d(2,1)=2,\\&d(0,2)=d(2,0)=3,\\
				%		\end{split}
			%	\end{equation*}
		%	and Let mapping $T$ on $X$ defined by
		%	\begin{equation*}
			%		T(1)=T(2)=2,\quad T(0)=1.
			%	\end{equation*}Then Suzuki contraction $T$ having a unique fixed point.
		%\end{example}
		%		
		\begin{remark}\cite{Suzuki 2008} Suzuki contraction is a significant mapping because this gives us metric completeness, i.e., $(S,d)$ is a complete metric space if and only if every Suzuki contraction mapping $T$ on  $(S,d)$ has a fixed point.
		\end{remark}  
		\begin{definition} \cite{Barnsley 2014}\label{D2}
			Let $\mathbb{C}(S)$ be a set of all non-void compact subsets of set $S$. The
			distance between any two compact subsets $A,B\in \mathbb{C}(S)$ is defined as:
			$D(A,B)=\sup\{d(a,B):a\in A \},$ where $d(a,B)= \inf\{d(a,b):a\in A,b\in B \}$. Then Hausdorff metric is defined as:
			\[
			h(A,B)= D(A,B) \vee D(B,A).
			\] Here, $\vee$ denotes maximum. The pair $(\mathbb{C}(S), h)$ is a Hausdorff-induced metric space of the underlying metric space (S,d). Hausdorff metric space is known as the `space of fractals '\cite {Barnsley 2014}. 
		\end{definition}
		\begin{lemma} \label{L1}
			There exists a mapping $T:\mathbb{C}(S)\to \mathbb{C}(S)$ defined by \begin{equation}\label{eq2}
				T(A)=\{T(a)|a\in A \}~\forall~ A\in \mathbb{C}(S),
			\end{equation}whenever $T:S\to S$ is a continuous Suzuki contraction mapping on the metric space $(S ,d)$.
		\end{lemma}
		\begin{proof}
			Whenever mapping $T$ is continuous then $T$ maps compact set to compact set so that $T$ is a well-defined mapping on $\mathbb{C}(S)$.
		\end{proof}
		\begin{definition} \label{D3}Let $(S,d)$ be complete metric space and Let $T_i$ be  Suzuki contraction mapping on  $(S,d)$ with respect to contractivity factor $m_i$ for $i=1,2,...,n$ where, $0\le m_i<1$. Then finite collection $\mathbb{I}=\{S; T_i, i\in\mathbb{N}_n\}$ is said to be a Suzuki iterated function system (SIFS).
		\end{definition}
		\begin{lemma}\cite{Barnsley 2014} \label{L2}
			If $\{A_i\}_{i=1}^{n}$ , $\{B_i\}_{i=1}^{n}$ are two finite collections of elements in $(\mathbb{C}(S))$ then
			\begin{equation*}
				h({\cup_{i=1}^{n} A_i},{\cup_{i=1}^{n}B_i})\leq\max h(A_i,B_i)
			\end{equation*}
		\end{lemma}
		
		\begin{lemma} \label{L3}
			Let $T$ be a self-mapping on the metric space $(S,d)$. If $T$ is a continuous Suzuki contraction on $(S,d)$ then $T$ is also a continuous Suzuki contraction on $(\mathbb{C}(S),h)$.
		\end{lemma}
		\begin{proof}
			By hypothesis, we have $T : S \rightarrow S$ is a Suzuki contraction mapping, i.e., $~\exists~ m \in \big[0,1\big) $ such that
			\begin{equation}
				\begin{split}
					Q(m)d(a,Ta)\le d(a,b)\Rightarrow d(Ta,Tb)\le m~d(a,b)~\forall~a,b\in S.
				\end{split}
			\end{equation}
			Our claim is to show that  $T : \mathbb{C}(S) \rightarrow \mathbb{C}(S)$ is a Suzuki contraction mapping, i.e., $~\exists~ r \in \big[0,1\big)$ such that
			\begin{equation*}
				Q(m)~h(A,TA)\le h(A,B)\Rightarrow h(TA,TB)\le r~h(A,B)~\forall~A,B\in \mathbb{C}(S).
			\end{equation*}
			Let $A,B\in \mathbb{C}(S)$ be arbitrary. By definition (\ref{D2}) for each $a$ fixed in $A$, and $z$ in $A$ we have
			\begin{equation*}
				\begin{split}
					\underset{z\in A}\inf~ d(a,Tz)\le d(a,Ta)
				\end{split}
			\end{equation*}
			Multiplying by $Q(m)$ on both sides of the above equation and followed by the hypothesis, we get
			\begin{equation*}
					Q(m)~\underset{z\in A}\inf~ d(a,Tz)\le Q(m)~d(a,Ta)\le d(a,b)
             \end{equation*}
             \begin{equation*}
         \Rightarrow d(Ta,Tb)\le m.d(a,b)~\forall~a~ fixed~in ~A, b\in B.
			\end{equation*}
			   Subsequently, by compactness of $B$
			\begin{equation*}
				\begin{split}
					Q(m)~\underset{z\in A}\inf~ d(a,Tz)\le \underset{b\in B}\inf ~d(a,b)=d(a,b_1)~for~some~ b_1 \in B\\ \Rightarrow \underset{b\in B}\inf~ d(Ta,Tb)\le d(Ta,Tb_1) \le m~\underset{b\in B}\inf~ d(a,b)~\forall~a~fixed~in~ A.
				\end{split}
			\end{equation*}
			Now by compactness of $A$
			\begin{equation*}
				Q(m)~\underset{a\in A}\sup~ \underset{z\in A}\inf ~d(a,Tz)=Q(m)~\underset{z\in A}\inf~d(\hat{a},Tz)\le\underset{b\in B}\inf ~d(\hat{a},b)\Rightarrow \underset{b\in B}\inf~ d(T\hat{a},Tb)\le m~\underset{b\in B}\inf ~d(\hat{a},b)
			\end{equation*}
			\begin{equation}
				Q(m)~\underset{a\in A}\sup ~\underset{z\in A}\inf~ d(a,Tz)=Q(m)~D(A,TA)\le\underset{b\in B}\inf~ d(\hat{a},b)\le D(A,B)\leq h(A,B)
			\end{equation}
			\begin{equation}\label{T1}
				\implies\underset{b\in B}\inf~ d(T\hat{a},Tb)\le m~\underset{b\in B}\inf ~d(\hat{a},b)\le m~D(A,B)\le m~h(A,B)
			\end{equation} \\
			%\textbf{Case I:} If the following inequality is valid then the result is obvious
			%\begin{equation*}
			%	\underset{b\in B}\inf ~d(T\hat{a},Tb)\le D(TA,TB)\le h(TA,TB)\le m~h(A,B)
			%\end{equation*} 
			%\textbf{Case II:}  
			%\begin{equation*}
			%	\underset{b\in B}\inf~ d(T\hat{a},Tb)\le D(TA,TB)\le m~h(A,B)\le h(TA,TB)
			%\end{equation*} Consider, there exist $K_{AB}\in \big(1,\infty\big)$ such that
			%\begin{equation}\label{T2}
			%	 \frac{1}{K_{AB}}h(TA,TB)\le m~h(A,B) \implies h(TA,TB)\le K_{AB}m~h(A,B)
			%	\end{equation}
		%Consider $k_1=\underset{AB\in\mathbb{C}(S)} \inf~\{K_{AB}\} $
		% such that $k_1m\in \big[0,1\big) \implies m\in\big[0,\frac{1}{k_1}\big)$, so that, suppose $k_1m=r_1$ that is $r_1 \in \big[0,1\big)$ and
		%by putting (\ref{T2}) in (\ref{T1}) we get,
		%\begin{equation}\label{A1}
		%	Q(m)~D(A,TA)\leq h(A,B)
		%	\implies h(TA,TB)\le r_1~h(A,B)\, \forall~A,B \in \mathbb{C}(S)
		%\end{equation}
		Similarly, we have
		Suppose, $A,B\in \mathbb{C}(S)$ be arbitrary. For each $a$ is fixed in $A$ and $z\in A$ we have
		\begin{equation*}
			\begin{split}
				\underset{z\in A}\inf~ d(Ta,z)\le d(Ta,a)
			\end{split}
		\end{equation*}
		multiplying by $Q(m)$ on both sides of the above equation and followed by hypothesis, we get
		\begin{equation*}
				Q(m)~\underset{z\in A}\inf ~d(Ta,z)\le Q(m)~d(Ta,a)\le d(a,b)
        \end{equation*}
        \begin{equation*}
        \Rightarrow d(Ta,Tb)\le m~d(a,b)~\forall~a~ is~ fixed~in ~A, b\in B.
		\end{equation*}
		Subsequently, By compactness of $B$
		\begin{equation*}
			\begin{split}
				Q(m)~\underset{z\in A}\inf ~d(Ta,z)\le~ \underset{b\in B}\inf~ d(a,b)=d(a,b_1)~for~some~ b_1 \in B\\ \Rightarrow ~ \underset{b\in B}\inf~ d(Ta,Tb)\le d(Ta,Tb_1) \le m~\underset{b\in B}\inf~ d(a,b)~\forall~a~fixed ~in~ A
			\end{split}
		\end{equation*}
		followed by compactness of $A$
		\begin{equation*}
			Q(m)~\underset{a\in A}\sup~ \underset{z\in A}\inf ~d(Ta,z)=Q(m)~d(Ta^\circ,z)\le\underset{b\in B}\inf ~d(a^\circ,b)\Rightarrow \underset{b\in B}\inf ~d(Ta^\circ,Tb)\le m~\underset{b\in B}\inf~ d(a^\circ,b)
		\end{equation*}
		\begin{equation*}
			Q(m)~\underset{a\in A}\sup~ \underset{z\in A}\inf ~d(Ta,z)=Q(m)~D(TA,A)\le\underset{b\in B}\inf ~d(a^\circ,b)=D(A,B)\leq h(A,B)
		\end{equation*}
		\begin{equation}\label{T2}
			\implies\underset{b\in B}\inf~ d(Ta^\circ,Tb)\le m~\underset{b\in B}\inf~ d(a^\circ,b)\leq m~D(A,B)\le m~h(A,B)
		\end{equation} 
         from equation (\ref{T1})and (\ref{T2}) we will write as follows
		\begin{equation*}
			Q(m)\max\{D(A,TA), D(TA,A)\}=Q(m)h(A,TA)\leq h(A,B)
		\end{equation*}
		\begin{equation}\label{T3}
			\implies\underset{b\in B}\inf~ d(Ta^\circ,Tb)\le m~h(A,B)
		\end{equation}
		Since $\underset{b\in B}\inf~ d(T\hat{a},Tb)\leq h(TA,TB)$ and $\underset{b\in B}\inf~ d(Ta^\circ,Tb)\leq h(TA,TB)$ where $\hat{a},a^\circ \in A$,~
		There are three cases showing inequality holds as follows
		\textbf{Case I:} For each disjoint element $ A, B \in \mathbb{C}(S)$
		\begin{equation*}
			h(TA,TB)\le m~h(A,B)
		\end{equation*} is holds for $m\in \big[0,1\big)$ then nothing to prove further.\\
		\textbf{Case II:} For each disjoint element $ A, B \in \mathbb{C}(S)$
		\begin{equation*}
			m~h(A,B)\le h(TA,TB)
		\end{equation*} there exist a constant $k$ such that
		\begin{equation}\label{T4}
			h(TA,TB)\le km~h(A,B)
		\end{equation}is holds for $m\in \big[0,\frac{1}{k}\big)$\\
		\textbf{Case III:} For some disjoint element $ A, B \in \mathbb{C}(S)$~ case I is true and for the rest case II true then the desired result is proved for $m\in\big[0,\frac{1}{k}\big),~~ k\in \big(1,\infty\big)$\\
		Thus,
		\begin{equation*} 
			Q(m)~h(A,TA)\le h(A,B)\Rightarrow h(TA,TB)\le m~h(A,B)~\forall~A,B\in \mathbb{C}(S)
		\end{equation*}
		holds Where, $m\in\big[0,\frac{1}{k}\big)$
		
		Hence the required proof.
	\end{proof}

\begin{theorem}\label{L4}
	let	$(S,d)$ be a metric space. Let $\{T_i : i = 1, 2,\hdots,n\}$ be collection of Suzuki contraction
	mappings on $(S,d)$. Let $m_i\in\big[0,1\big)$ be contractivity factor of $T_i$, for each $i$.\\ Define \textbf{Hutchinsun map} $\mathbb{T}:(\mathbb{C}(S),h)\to(\mathbb{C}(S),h)$ by
	\begin{equation*}
		\begin{aligned}
			\mathbb{T}{A}&={T_1{A}}\cup {T_2}{A}\cup {T_3{A}}\cup...\cup {T_n}{A}\\ & =\bigcup_{i=1}^{n} {T_i}{A} \quad \forall~ A \in \mathbb{C}(S)
		\end{aligned}
	\end{equation*}
	Then $\mathbb{T}$ is Suzuki contraction mapping with contractivity factor $r=\max\{r_1, r_2,\hdots,r_n\}$.
\end{theorem}
\begin{proof}
	By hypothesis, we have $T_{i} : S \rightarrow S$ is a Suzuki contraction mapping, i.e., there exist $r_{i}$ such that
	\begin{equation}
		\begin{split}
			Q(m)d(a,T_{i}a)\le d(a,b)\Rightarrow d(T_{i}a,T_{i}b)\le m_{i}~d(a,b)~\forall~a,b\in S.
		\end{split}
	\end{equation} Suppose $A,\,B\in\mathbb{C}(S)$ then by above lemma $\ref{L3}$\\ We have
	\begin{equation*} 
		Q(m)~h(A,T_{i}A)\le h(A,B)\Rightarrow h(T{i}A,T{i}B)\le r_{i}~h(A,B)~\forall~A,B\in \mathbb{C}(S)
	\end{equation*} where $r_{i}=m_{i}k_{i}$,  $k_{i}\geq\max\{\frac{h(T_{i}A,T_{i}B)}{m_{i}h(A,B)}: A, B \in S\}$.
	\begin{equation*}
		\implies h(T_iA,T_iB)\leq\max\{r_i\}~h(A,B)
	\end{equation*} for each $i$ 
	\begin{equation}
		\implies\max h(T_iA,T_iB)\le \max \{r_i\}~h(A,B)
	\end{equation}by lemma $\ref{L2}$

	\begin{equation}\label{T5}
		\begin{split}
			Q(m)~h(A,\mathbb{T}A)&=Q(m)~h(A,\bigcup_{i=1}^{n}T_iA)\le Q(m)\max\{h(A,T_iA)\},\\&\leq Q(m)~h(A,T_{i^*}A))\le h(A,B)
		\end{split}
	\end{equation}
	\begin{equation*}
		\implies \max h(T_iA,T_iB)\le \max \{r_i\}~h(A,B)
	\end{equation*}  
	\begin{equation}
		\implies h(\mathbb{T}A,\mathbb{T}B)\le \max \{r_i\}~h(A,B)
	\end{equation} therefore
	\begin{equation}
		Q(m)~h(A,\mathbb{T}A)\le h(A,B)\implies h(\mathbb{T}A,\mathbb{T}B)\le m^{*}~h(A,B)
	\end{equation} where $m^{*}=\max \{r_i\}$
	Hence the required proof.
\end{proof}
\begin{remark}
	\begin{example}
		Let $S=\big[0,1\big]\bigcup \{4,5,8,9,\hdots,4n,4n+1,\hdots\}$ is metric space with usual metric $d$ and 
		\begin{equation*}
			T(x)= \left\{
			\begin{array}{ll}
				\frac{x}{3}&if\,\qquad x\in \big[0,1\big], \\
				0&if \,\quad~~~~x=4n, \\
				1-\frac{1}{n+3}&if \,\qquad x=4n+1.\\
			\end{array}\right
			.\end{equation*}
		The function $T$ is suzuki contraction on $(S,d)$ then $T$ is Suzuki contraction on $(\mathbb{C}(S), h)$.
	\end{example}
\end{remark}
\begin{lemma} \label{L5}\cite{AKBC}
	If $(S,d)$ is complete metric space then $(\mathbb{C}(S),h)$ is also complete metric space.
\end{lemma}

\begin{theorem} \label{TH1}
	Let $\mathbb{I}$ be a SIFS with contractivity factor $r$ and Hutchinsun mapping $\mathbb{T}:\mathbb{C}(S)\to\mathbb{C}(S)$ is a Suzuki contraction mapping on the complete metric space $(\mathbb{C}(S),h)$ with contractivity factor $r$.
	Then $\mathbb{T}$ has a unique fixed point  $F\in\mathbb{C}(S)$ obeys the self referential equation \begin{equation*}
		F=\mathbb{T}{F}=\bigcup_{i=1}^{n} {T_i}(F)\quad such~that \quad
		\lim_{n\to\infty}\mathbb{T}^n A=F\quad \forall~ A\in \mathbb{C}(S)
	\end{equation*}
\end{theorem}
\begin{proof}
	Let  $ (S,d )$ be a complete metric space then we have $(\mathbb{C}(S),h)$ is also a complete metric space. Since $T$ is suzuki contraction on $(S,d)$ and from lemma (\ref{L3}) $T$ is suzuki contraction on $(\mathbb{C}(S),h)$ with same contractivity factor $m$. Hutching map $\mathbb{T}$ is also suzuki contraction on  $(\mathbb{C}(S),h)$ with contractivity factor $m^*=\max\{m_1,m_2,\hdots,m_n\}$ by lemma (\ref{L4}).\\
	,$(\mathbb{C},h)$ is complete metric space and  Hutching map $\mathbb{T}$ is also suzuki contraction on  $(\mathbb{C},h)$ with contractivity factor $m^*=\max\{m_1,m_2,...,m_n\}$ then by Suzuki contraction principle,  $\mathbb{T}$ having unique fixed point $F$ and $ \lim_{t\to\infty}\mathbb{T}^t{A}=F\quad \forall~ {A}\in \mathbb{C}(S) $.
\end{proof}
\begin{remark}
	$F$ is said to be a fixed point (attractor or deterministic fractal \cite{Barnsley 2014}) of Hutchinson mapping $\mathbb{T}$.
\end{remark}
	\section{Conclusions}
	We presented a new type of non-linear contraction in this study called Suzuki contraction on hyperspace, which is a more general Banach contraction because many continuous functions are not Banach contractions but Suzuki contractions. For Suzuki contraction mapping in hyperspace, we demonstrated the existence and uniqueness of a fixed point. We constructed new IFSs based on Suzuki contractions, dubbed Suzuki IFS for use in fractal constructions, which are tight generalizations of the conventional Hutchinson-Barnsley theory of IFS. We also confirmed the presence of attractors for these IFS and their uniqueness. Finally, we got generalized Suzuki IFS, which is not possible in the classical approach to metric space.

\end{document}